\documentclass{amsart}

\newtheorem{theorem}{Theorem}[section]
\newtheorem{proposition}[theorem]{Proposition}
\newtheorem{lemma}[theorem]{Lemma}
\newtheorem{corollary}[theorem]{Corollary}
\newtheorem{fact}[theorem]{Fact}

\theoremstyle{definition}
\newtheorem{definition}[theorem]{Definition}

\theoremstyle{remark}
\newtheorem{remark}[theorem]{Remark}

\numberwithin{equation}{section}

\def \a {\bar a}
\def \b {\bar b}
\def \c {\bar c}
\def \x {\bar x}
\def \y {\bar y}

\def \z {\zeta}

\def \T {\bar t}
\def \d {\delta}
\def \dd {\partial}
\def \D {\Delta}
\def \t {\theta}
\def \T {\Theta}

\def \I {\mathcal I}

\def \Z {\mathcal V}

\def \C {\mathcal C}
\def \NN {\mathbb N}

\def \QQ {\mathbb Q}
\def \UU {\mathbb U}
\def \LL {\mathcal L}
\def \KK {\mathcal K}

\title{}
\author{}

\begin{document}

\begin{center}
{\large \bfseries GEOMETRIC AXIOMS FOR DIFFERENTIALLY CLOSED FIELDS WITH SEVERAL COMMUTING DERIVATIONS}

\vspace{5mm}

{\large Omar Le\'on S\'anchez}\\
University of Waterloo, ON, Canada \\
March 3, 2011

\end{center}

\begin{abstract}
A geometric first-order axiomatization of differentially closed fields of characteristic zero with several commuting derivations, in the spirit of Pierce-Pillay \cite{PiPi}, is formulated in terms of a relative notion of prolongation for Kolchin-closed sets.
\end{abstract}

\maketitle

\begin{center}
\end{center}

\

\section{Introduction}

An ordinary differential field is a field of characteristic zero equipped with a derivation, that is, an additive map $\d:K\to K$ such that $\d(ab)=(\d a) b+a(\d b)$. A differentially closed field is a differential field $(K,\d)$ such that every system of differential polynomial equations in several variables, with a solution in some differential extension, has a solution in $K$. An elegant first-order axiomatization of the class of ordinary differentially closed fields was given by Blum in \cite{Bl}. In \cite{PiPi}, Pierce and Pillay give a geometric axiomatization. Their axioms say that $(K,\d)$ is differentially closed if and only if $K$ is algebraically closed and whenever $V$ and $W$ are irreducible Zariski-closed sets with $W$ contained in the prolongation of $V$ and projecting dominantly onto $V$, then there is a $K$-point in $W$ of the form $(\x,\d\x)$.

Similarly, a field $K$ of characteristic zero equipped with $m$ commuting derivations is differentially closed if every system of partial differential polynomial equations in several variables with a solution in some extension has a solution in $K$. A first-order axiomatization generalizing Blum's was given by McGrail in \cite{Mc} \footnote{A different algebraic axiomatization can be found in Tressl \cite{Tr}.}. However, if one tries to give geometric axioms in terms of prolongations, the commutativity of the derivations might impose too many restrictions, so that a Pierce-Pillay type condition will not hold (see \cite{Pie}, Counterexample 6.2). Nonetheless, in \cite{Pie}, Pierce does manage to give an axiomatization (in arbitrary characteristic) that has a geometric flavor, though not exactly in the Pierce-Pillay sense. In this paper we take a different approach, establishing an axiomatization of differentially closed fields with $(m+1)$ commuting derivations which is geometric relative to the theory of differentially closed fields with $m$ derivations. Our axioms are a precise generalization of the Pierce-Pillay axioms. Two complications arise in our setting that do not appear in the ordinary case: one has to do with extending commuting derivations and the other has to do with first-order axiomatizability. Differential-algebraic results due to Kolchin are behind our solutions to both of these problems.

Suppose $\D=\{\d_1,\dots,\d_m\}$ are commuting derivations on a field $K$ of characteristic zero and for each $r=0,\dots, m$, let $\D_r=\{\d_1,\dots,\d_r\}$. Also, suppose $D:K\to K$ is an additional derivation on $K$ that commutes with $\D$. If $V$ is a $\D$-closed set defined over the $D$-constants of $K$, then Kolchin constructs a $\D$-tangent bundle of $V$ which has $\x\to(\x,D\x)$ as a section (\cite{Kol}, Chap. VIII, \S 2). In general, if $V$ is not necessarily defined over the $D$-constants, then $D$ gives a section of a certain torsor of the $\D$-tangent bundle of $V$ that we call the \emph{$D/\D$-prolongation} of $V$ (cf. Definition \ref{prolong}).  Our axioms will essentially say that {\it $(K,\D\cup\{D\})$ is differentially closed if and only if $K$ is algebraically closed and for each $r=0,\dots,m$, whenever $V$ and $W$ are $\D_r$-closed sets with $W$ contained in the $D/\D_r$-prolongation of $V$ and projecting onto $V$, then there is a $K$-point in $W$ of the form $(\x,D\x)$}. ``Essentially'', because in actual fact we also have to consider not just $\D$ and $D$ but all their independent $\QQ$-linear combinations (cf. Theorem \ref{main2} below).

The paper is organized as follows. In Section 2 we establish the differential-algebraic facts that underpin our results. In Section 3 we introduce relative prolongations and prove a geometric characterization of differentially closed fields. Finally, in Section 4, we address the issue of first-order axiomatizability.

\

\noindent{\it Acknowledgements:} I would like to thank Rahim Moosa for all the useful discussions and support towards the completion of this article.

\section{Extending $\D$-derivations}

In this paper the term ring is used for commutative ring with unity and the term field for field of characteristic zero.

Let us first recall some terminology from differential algebra. For details see \cite{Ko}. Let $R$ be a ring and $S$ a ring extension. An additive map $\d:R\to S$ is called a derivation if it satifies the Leibniz rule; i.e., $\d(ab)=(\d a)b+a(\d b)$. A ring $R$ equipped with a set of derivations $\D=\{\d_1,\dots,\d_m\}$, $\d_i:R\to R$, such that the derivations commute with each other is called a $\D$-ring. A $\D$-ring which is also a field (of characteristic zero) is called a $\D$-field. 

We fix for the rest of this section a $\D$-ring $R$. Let $\T$ denote the free commutative monoid generated by $\D$; that is, 
\begin{displaymath}
\T:=\{\d_m^{r_m}\cdots\d_1^{r_1}\,:\,r_m,\dots,r_1\geq 0\}.
\end{displaymath}
The elements of $\T$ are called the derivative operators of $R$. Let $\x=(x_1,\dots,x_n)$ be a family of indeterminates, and define
\begin{displaymath}
\t\x:=\{\partial x_j: \, j=1,\dots,n, \, \partial \in \T\}.
\end{displaymath}
The $\D$-ring of $\D$-polynomials over $R$ in the differential indeterminates $\x$ is $R\{\x\}:=R[\t\x]$; that is, the ring of polynomials in the algebraic indeterminates $\t\x$ with the canonical $\D$-ring structure given by $\d_i(\d_m^{r_m}\cdots\d_1^{r_1}x_j)=\d_m^{r_m}\cdots\d_i^{r_i+1} \cdots\d_1^{r_1}x_j$. 

We fix an orderly ranking in $\t\x$ by:
\begin{displaymath}
\d_m^{r_m}\cdots\d_1^{r_1}x_i\leq \d_m^{r'_m}\cdots\d_1^{r'_1}x_j \iff \left(\sum r_l,i,r_m,\dots,r_1\right)\leq \left(\sum r'_l,j,r'_m,\dots,r'_1\right)
\end{displaymath}
in the lexicographical order. According to this ranking, we enumerate the algebraic indeterminates by $\t \x=(\t_1\x,\t_2\x,\dots)$. Therefore, if $f\in R\{\x\}$ there is a unique $\hat f\in R[t_1,t_2,\dots]$ such that $f(\x)=\hat f(\t\x)$. 

We will be interested in adding an extra derivation on $R$. This amounts to the study of $\D$-derivations (see Chapter 0 of \cite{Kol}). 
\begin{definition}
A $\D$-derivation on $R$ is a derivation $D:R\to S$, where $S$ is a $\D$-ring extension of $R$, such that $D\d_i=\d_iD$ for $i=1,\dots,m$.
\end{definition}

If $f\in R\{\x\}$, by $f^D$ we mean the $\D$-polynomial in $S\{\x\}$ obtained by applying $D$ to the coefficients of $f$. 

\begin{remark}\label{super} \ 
\begin{enumerate} 
\item The map $f\mapsto f^D$ is itself a $\D$-derivation from $R\{\x\}$ to $S\{\x\}$. Indeed, it is clearly additive and to check the Leibniz rule and commutativity of $D$ with $\D$ it suffices to show that $(fg)^D=f^Dg+fg^D$ and $(\d_i f)^D=\d_i(f^D)$ for $i=1,\dots,m$, where $f(\x)=\hat f(\t\x)$ and $g(\x)=\hat g(\t\x)$ are such that $\hat f$ and $\hat g$ are monomials in $R[t_1,t_2,\dots]$. These computations are straightforward.
\item The $\D$-derivation $D$ extends uniquely to a derivation 
\begin{displaymath}
R\{\x\}\to S[D\t\x ]:= S[D\t_k\x:\, k=1,\dots]
\end{displaymath}
by $D(\t_k\x)=D\t_k\x$, for $k\geq 1$. Also, each $\d_i$ extends uniquely to a derivation
\begin{displaymath}
S[D\t \x]\to S[\d_iD\t \x] :=S[\d_iD\t_k\x:\, k=1,\dots]
\end{displaymath}
by $\d_i(D\t_k\x)=\d_iD\t_k\x$, for $k\geq 1$.
\end{enumerate}
\end{remark}

In order to describe how $D$ and compositions of $D$ with elements of $\D$ act on $R\{\x\}$, we introduce the following convenient terminology. Given $f\in R\{\x\}$, the \emph{Jacobian} of $f$ is 
\begin{displaymath}
df(\x):=\left( \frac{\partial \hat f}{\partial t_i}(\t\x)\right)_{i\in \NN}
\end{displaymath}
viewed as an element of $\left(R\{\x\}\right)^{\NN}$. Note that $df$ is finitely supported, in the sense that all but finitely many coordinates are zero. A straightforward computation shows that for each $\d_i\in \D$
\begin{displaymath}
\d_if(\x)=df(\x)\cdot\d_i\t \x+f^{\d_i}(\x),
\end{displaymath}
where $\d_i\t\x=(\d_i\t_1\x,\d_i\t_2\x,\dots)$ and the dot product is well defined since $df$ has finite support. 

The \emph{Hessian} of $f$ is defined as
\begin{displaymath}
Hf(\x):=\left(\frac{\partial^2 \hat f}{\partial t_i\partial t_j}(\t\x)\right)_{i,j\in \NN}
\end{displaymath} 
viewed as an element of $\left(R\{\x\}\right)^{\NN\times\NN}$. Again $Hf$ is finitely supported.

\begin{lemma}\label{exten1}
Let $S$ be a $\D$-ring extension of $R$ and $D:R\to S$ a $\D$-derivation. If $f\in R\{\x\}$ and $\d\in \D\cup \{D\}$, then
\begin{displaymath}
\d f(\x)=df(\x)\cdot \d \t\x + f^{\d}(\x).
\end{displaymath}
Also, if $\d$, $\z \in\D\cup \{D\}$ with $\d \neq \z$, then
\begin{eqnarray*}
\d\z f(\x)
&=& df(\x)\cdot\d\z\t\x+\d \t\x \cdot Hf(\x)\cdot (\z \t\x)^t + f^{\d\z}(\x) \\
&& + \, df^{\d}(\x)\cdot\z\t\x + df^{\z}(\x)\cdot \d\t\x. \\
\end{eqnarray*}
For any choice of $\d$ and $\z$ the elements $\d f$ and $\d\z f$ are well defined by part (2) of Remark \ref{super}. Note that the dot product $df(\x)\cdot\d\t\x$ and the matrix product $\d \t\x \cdot Hf(\x)\cdot (\z \t\x)^t$ are well defined because $df$ and $Hf$ have finite support.
\end{lemma}
\begin{proof}
For the first equation, by additivity, it suffices to prove it for monomials $f(\x)=c\prod_{i}(\t_i \x)^{\alpha_i}$ where $c\in R$ and $\alpha_i=0$ except finitely many times:
\begin{eqnarray*}
\d f(\x) 
&=& \sum_{k}\left( c\prod_{i\neq k} (\t_i \x)^{\alpha_i} \; \alpha_k (\t_k \x)^{\alpha_k-1} \d \t_k \x \right)+ \d(c)\prod_{i}(\t_i \x)^{\alpha_i} \\
&=& df(\x)\cdot \d \t \x + f^{\d}(\x). \\
\end{eqnarray*}
For the second equation,
\begin{eqnarray*}
\d \z f(\x)
&=& \d\left( df(\x)\cdot \z \t\x + f^{\z}(\x)\right) \\
&=& \d\left(\sum_k \frac{\partial \hat f}{\partial t_k}(\t \x)\z\t_k \x+f^{\z}(\x)\right) \\
&=& \sum_k \frac{\partial \hat f}{\partial t_k}(\t \x)\d\z\t_k \x+ \sum_k \d\left(\frac{\partial \hat f}{\partial t_k}(\t \x)\right)\z\t_k \x+\d\left(f^{\z}(\x)\right)\\
&=& df(\x)\cdot\d\z\t\x+ \sum_k\left(d\left(\frac{\partial \hat f}{\partial t_k}(\t\x)\right)(\x)\cdot \d\t\x+ \left(\frac{\partial \hat f}{\partial t_k}(\t \x)\right)^\d (\x)\right)\z \t_k \x\\
&& + \, df^{\z}(\x)\cdot\d\t \x+f^{\d\z}(\x)\\
&=& df(\x)\cdot\d\z\t\x+\sum_k\sum_l\frac{\partial^2 \hat f}{\partial t_l\partial t_k}(\t \x)\, \d\t_l \x \; \z\t_k \x \\
&& + \sum_k \left(\frac{\partial \hat f}{\partial t_k}(\t \x)\right)^{\d}(\x) \, \z\t_k \x+	df^{\z}(\x)\cdot\d\t\x+f^{\d\z}(\x)\\
&=& df(\x)\cdot\d\z\t\x+\d\t\x\cdot Hf(\x)\cdot (\z\t\x)^t +f^{\d\z}(\x)\\
&& + \sum_k \left(\frac{\partial {\widehat {f^{\d}}}}{\partial t_k}\right)(\t \x)\z\t_k \x+	df^{\z}(\x)\cdot\d\t\x\\
&=& df(\x)\cdot\d\z\t\x+\d\t\x\cdot Hf(\x)\cdot (\z\t\x)^t +f^{\d\z}(\x)\\
&& + \, df^{\d}(\x)\cdot\z\t\x+	df^{\z}(\x)\cdot\d\t\x.\\
\end{eqnarray*}
Where the sixth equality uses $\left(\frac{\partial \hat f}{\partial t_k}(\t\x)\right)^{\d}=\frac{\partial \widehat{f^{\d}}}{\partial t_k}(\t\x)$. This follows by additivity and the fact that
\begin{eqnarray*}
\left[\frac{\partial (c \, t_1^{n_1}\cdots t_l^{n_l})}{\partial t_k}(\t\x)\right]^\d
&=& \left( n_k c \, (\t_1\x)^{n_1}\cdots (\t_k \x)^{n_k-1}\cdots (\t_l \x)^{n_l} \right)^\d \\
&=& n_k\d(c) \, (\t_1\x)^{n_1}\cdots (\t_k \x)^{n_k-1}\cdots (\t_l \x)^{n_l}\\
&=& \frac{\partial (\d(c)t_1^{n_1}\cdots t_l^{n_l})}{\partial t_k}(\t\x).\\
\end{eqnarray*}
\end{proof}

\begin{corollary}\label{exten2} 
Let $S$ be a $\D$-ring extension of $R$ and $D:R\to S$ a $\D$-derivation. Suppose $\a$ is a tuple in $S$ and $D':R\{\a\}\to S$ is a derivation extending $D$ such that $D'\t\a=\t D'\a$. Then $D'$ commutes with $\D$ on $R\{\a\}$. 
\end{corollary}
\begin{proof}
Let $f\in R\{\x\}$ we must show that for any $\d_i\in \D$, $\d_iD' f(\a)=D'\d_if(\a)$. By the second equality of Lemma \ref{exten1}, we have that
\begin{displaymath}
\d_iD' f(\a)-D'\d_if(\a)=df(\a)\cdot(\d_iD'\t\a-D'\d_i\t\a).
\end{displaymath} 
But by assumption $\d_iD'\t\a=\d_i \t D'\a=D'\d_i\t \a$, as desired.
\end{proof}

\begin{definition}\label{tauf}
Let $f\in R\{\x\}$. We define the $\D$-polynomial $\tau_{D/\D} f\in S\{\x,\y\}$ by
\begin{displaymath}
\tau_{D/\D} f(\x, \y):= df(\x)\cdot \t \y +f^D(\x).
\end{displaymath}
When $\D$ and $D$ are understood we simply write $\tau f$. If $\a\in S$, we write $\tau(f)_{\a}(\y)$ for $\tau f(\a,\y)\in S\{\y\}$. Note that $\tau \t \x=\t \y$ and if $c\in R$ then $\tau c=Dc$. 
\end{definition}

\begin{lemma}\label{exten3}
Suppose $S$ is a $\D$-ring extension of $R$ and $D:R\to S$ is a $\D$-derivation. Then $\tau:R\{\x\}\to S\{\x,\y\}$ is a $\D$-derivation extending $D$.
\end{lemma}
\begin{proof}
The map $f\mapsto \tau f$ is additive and, by part (1) of Remark \ref{super} and the fact that $d(fg)(\x)\cdot\t\y=\left(df(\x)\cdot\t\y\right) g(\x)+f(\x)\left(dg(\x)\cdot\t\y\right)$, we have that $\tau(fg)=(\tau f)g+f(\tau g)$.
Hence, since $\tau c=D c$ for all $c\in R$, $\tau$ is a derivation extending $D$.
For commutativity, note that $\tau\t \x=\t \y=\t \tau \x$. Now just apply Corollary \ref{exten2} with $\tau$, $\x$ and $S\{\x,\y\}$ in place of $D'$, $\a$ and $S$.
\end{proof} 

We can now give the desired criterion for when a $\D$-derivation can be extended to a finitely generated $\D$-ring extension. The analogue of the following proposition when $\D=\emptyset$ (i.e. $m=0$) can be found in (\cite{La}, Chap. VII, \S5), and it is the main point in the Pierce-Pillay geometric axiomatization of ordinary differentially closed fields.

\begin{proposition}\label{exten5}
Suppose $S$ is a $\D$-ring extension of $R$ and $D:R\to S$ is a $\D$-derivation. Let $\a$ be a tuple of $S$ and $A\subseteq R\{\x\}$ such that $[A]=\I(\a/R)$. Here $[A]$ denotes the $\D$-ideal generated by $A$ and $\I(\a/R)=\{f\in R\{\x\} : f(\a)=0\}$. Suppose there is a tuple $\b$ of $S$ such that
\begin{equation}\label{form}
\tau (g)_{\a}(\b)=0, \textrm{ for all } g\in A.
\end{equation}
Then there is a unique $\D$-derivation $D':R\{\a\}\to S$ extending $D$ such that $D'\a=\b$. 
\end{proposition}
\begin{proof}
First we show that (\ref{form}) holds for all elements in $\I(\a/R)$. For each $\dd\in \T$, $g\in A$ and $h\in R\{\x\}$, by Lemma \ref{exten3} we have
\begin{equation}\label{eq11}
\tau (h\dd g)(\x,\y)=\tau h(\x,\y)\dd g(\x)+h(\x)\dd(\tau g(\x,\y)).
\end{equation}
By assumption $\tau g(\a,\b)=\tau (g)_{\a}(\b)=0$ and since $g\in \I(\a/R)$ we get $\dd g(\a)=0$, so evaluating (\ref{eq11}) at $(\a,\b)$ yields $\tau(h\dd g)_{\a} (\b)=0$. It follows that for each $f\in[A]=\I(\a/R)$, $\tau(f)_{\a} (\b)=0$.

Now let $\alpha\in R\{\a\}$, then $\alpha=f(\a)$ for some $f\in R\{\x\}$. Define
\begin{displaymath}
D'\alpha=\tau(f)_{\a} (\b).
\end{displaymath}
This does not depend on the choice of $f$, since if $\alpha=f(\a)=f'(\a)$ then $f-f'\in \I(\a/R)$ and hence by assumption $\tau(f)_{\a}(\b)-\tau (f')_{\a}(\b)=\tau(f-f')_{\a}(\b)=0$. It is clear that $D'\a=\tau(\x)_{\a}(\b)=\b$ and, for all $c\in R$, $D'c=\tau(c)_{\a}(\b)=Dc$. By Lemma ~\ref{exten3} this is a $\D$-derivation extending $D$ to $R\{\a\}\to S$. To show uniqueness suppose $D'':R\{\a\}\to S$ is another $\D$-derivation extending $D$ such that $D''\a=\b$. Then, for any $f\in R\{\x\}$, by the first equality of Lemma \ref{exten1} we get
\begin{eqnarray*}
D'' f(\a)
&=& df(\a)\cdot D''\t\a+f^{D''}(\a)=df(\a)\cdot \t D''\a+f^{D}(\a)\\
&=& df(\a)\cdot \t\b+f^{D}(\a)=\tau(f)_{\a} (\b)=D'f(\a)
\end{eqnarray*}
so $D'=D''$.
\end{proof}

Thus if we want to extend a $\D$-derivation we need to find solutions to certain $\D$-equations. The following result of Kolchin's can then be used to see that there is always such a solution in some $\D$-field extension.

\begin{fact}[\cite{Kol}, Chap. 0, \S 4]\label{pro1}
Suppose $R$ is a $\D$-subring of a $\D$-field $K$ and $D:R\to K$ a $\D$-derivation. Let $\a$ be a tuple of $K$. Then 
\begin{displaymath}
[\tau (f)_{\a}(\y):f\in \I(\a/R)]
\end{displaymath}
is a proper $\D$-ideal of $K\{\y\}$.
\end{fact}

Putting Proposition \ref{exten5} and Fact \ref{pro1} together we get the following result on extending $\D$-derivations to differentially closed fields, which is essentially Corollary 1 of \S 0.4 of Kolchin \cite{Kol}.

\begin{corollary}\label{Null}
Suppose $R$ is a $\D$-subring of a differentially closed field $(K,\D)$, and $D:R\to K$ a $\D$-derivation. Then there is a $\D$-derivation $D':K\to K$ extending $D$.
\end{corollary}
\begin{proof}
Consider the set
\begin{eqnarray*}
\mathcal S 
&=& \{(S,D'): R\subseteq S\subseteq K \textrm{ is a $\D$-subring of $K$ and} \\ 
&& \textrm{ $D':S\to K$ is a $\D$-derivation extending $D$}\}.
\end{eqnarray*}
partially ordered by containment. The union of any chain of $\mathcal S$ gives an upper bound, and so, by Zorn's lemma, there is a maximal element, call it $(S,D')$. Towards a contradiction, suppose $S\neq K$. Then there is $a\in K\backslash S$, and, by Fact \ref{pro1}, $[\tau (f)_{a}(y):f\in \I(a/S)]$ is a proper $\D$-ideal of $K\{\x\}$. By the Nullstellensatz for differentially closed fields (see for example theorem 3.1.10 of \cite{Mc}), there is $b\in K$ such that $\tau(f)_a(b)=0$ for all $f\in \I(a/S)$. By Proposition \ref{exten5}, there is a $\D$-derivation $D'':S\{a\}\to K$ extending $D'$. This contradicts the maximality of $(S,D')$. Hence $S=K$ and $D':K\to K$ is a $\D$-derivation extending $D$.
\end{proof}

We conclude this section with an improvement on Proposition \ref{exten5}. We would like to only have to check condition (\ref{form}) for a set of $\D$-polynomials $A\subseteq R\{\x\}$ such that $\{A\}=\I(\a/R)$, where $\{A\}$ denotes the radical $\D$-ideal generated by $A$. As the reader may expect this will be useful when dealing with issues of first-order axiomatizability (see Proposition \ref{prolim} below).

First we need a lemma. For each $i=1,2,\dots$, let $\x_i$ be an $n$-tuple of differential indeterminates. Suppose $D:R\to R$ is a $\D$-derivation. Then $\tau:~R\{\x_1\}\to R\{\x_1,\x_2\}$. Thus we can compose $\tau$ with itself, for each $k\geq 1$ and $f\in R\{\x_1\}$, $\tau^k f=\tau\cdots \tau f\in R\{\x_1,\x_2,\dots,\x_{2^k}\}$. Define $\nabla\x:=(\x,D\x)$ and note that, for each $k\geq 1$, the composition $\nabla^k\x=\nabla\cdots\nabla\x$ is a tuple of length $n2^{k}$.

\begin{lemma}\label{radic}
Suppose $D:R\to R$ is a $\D$-derivation and $f\in R\{\x\}$. 
\begin{enumerate}
\item If $\a$ is a tuple of $R$, then for each $k\geq 1$,
\begin{displaymath}
\tau^kf(\nabla^k \a)=D^kf(\a)
\end{displaymath}
In particular, if $f(\a)=0$ then $\tau^k f(\nabla^k \a)=0$.
\item For each $k\geq 1$, we have 
\begin{displaymath}
\tau^k f^k=k! (\tau f)^k+f\, p
\end{displaymath}
for some $p\in R\{\x_1,\x_2,\dots,\x_{2^k}\}$.
\end{enumerate}
\end{lemma}
\begin{proof}
(1) By induction on $k$. The first equality of Lemma \ref{exten1} gives us
\begin{displaymath}
\tau f(\nabla \a)=df (\a)\cdot \t D\a+f^D(\a)=df (\a)\cdot D \t \a+f^D(\a)=Df(\a).
\end{displaymath}
The induction step follows easily: 
\begin{displaymath}
\tau^{k+1} f(\nabla^{k+1}\a)=\tau(\tau^k f)(\nabla(\nabla^k\a))=D\tau^k f(\nabla^k\a)=DD^k f(\a)=D^{k+1}f(\a).
\end{displaymath}
(2) We prove that for each $l=1,\dots,k$ we have 
\begin{equation}\label{good}
\tau^l(f^k)=\frac{k!}{(k-l)!}f^{k-l}(\tau f)^l+f^{k-l+1} \, p_l
\end{equation}
for some $p_l\in K\{\x_1,\x_2,\dots,\x_{2^l}\}$. From which the results follows when $l=k$. By Lemma \ref{exten3}, we have $\tau f^k=kf^{k-1}\tau f$, so (\ref{good}) holds for $l=1$ with $p_1=0$. Assume it holds for $1\leq l <k$, then
\begin{eqnarray*}
\tau^{l+1}f^k
&=& \tau\tau^l f^k=\tau \left(\frac{k!}{(k-l)!}f^{k-l}(\tau f)^l+f^{k-l+1} p_l\right)  \\
&=&\frac{k!}{(k-l)!}\left( (k-l)f^{k-l-1}(\tau f)^{l+1}+lf^{k-l}(\tau f)^{l-1}\tau^2 f \right)\\
&& + \; (k-l+1)f^{k-l}(\tau f) \, p_l+ f^{k-l+1}\tau p_l\\
&=& \frac{k!}{(k-l-1)!}f^{k-l-1}(\tau f)^{l+1}+f^{k-l} \,p_{l+1}
\end{eqnarray*}
where
\begin{displaymath}
p_{l+1}=\frac{k!\; l}{(k-l)!}(\tau f)^{l-1}\tau^2 f+(k-l+1)(\tau f) \, p_l+ f \tau p_l.
\end{displaymath}
\end{proof}

\begin{proposition}\label{better}
Suppose $R$ is a reduced $\QQ$-algebra and $D:R\to R$ is a $\D$-derivation. Let $\a$ a tuple of $R$ and $A\subseteq \I(\a/R)$. Suppose there is a tuple $\b$ of $R$ such that
\begin{equation}\label{rad}
\tau (g)_{\a}(\b)=0, \textrm{ for all } g\in A.
\end{equation}
Then $\tau (f)_{\a}(\b)=0$ for all $f\in \{A\}$. 
\end{proposition}
\begin{proof}
By the first argument in the proof of Proposition \ref{exten5}, equation (\ref{rad}) holds for all elements in $[A]$. Let $f\in \{A\}$, since $R\{\x\}$ is also a $\QQ$-algebra $\{A\}=\sqrt{[A]}$, and so there is $k\geq 1$ such that $f^k\in [A]$ and hence $\tau f^k(\a,\b)=0$. By part (1) of Lemma \ref{radic}, $\tau^{k-1}(\tau f^k)(\nabla^{k-1}(\a,\b))=0$. Thus, by part (2) of Lemma \ref{radic}, we have
\begin{displaymath}
k!(\tau f)^k(\a,\b)+f(\a)p(\nabla^{k-1}(\a,\b))=0,
\end{displaymath}
for some $p\in R\{\x_1,\x_2,\dots,\x_{2^k}\}$. Since $f(\a)=0$, we get $k!(\tau f)^k(\a,\b)=0$. Thus, since $R$ is a reduced $\QQ$-algebra, $\tau (f)_{\a}(\b)=0$.
\end{proof}

\begin{corollary}
If $S$ is a $\D$-field, then Proposition \ref{exten5} holds even if we replace the assumption that $[A]=\I(\a/R)$ by $\{A\}=\I(\a/R)$.
\end{corollary}
\begin{proof}
Suppose $\{A\}=\I(\a/R)$ and $\tau_{D/\D}(g)_{\a}(\b)=0$ for all $g\in A$. Let $(K,\D)$ be a differentially closed field extending $S$. By Corollary \ref{Null}, we can extend $D$ to a derivation $D':K \to K$. Now, by Proposition \ref{better}, $\tau_{D'/\D}(g)_{\a}(\b)=0$ holds for all $g\in \{A\}_K$, where $\{A\}_K$ denotes the radical $\D$-ideal in $K\{\x\}$ generated by $A$. But $\{A\}\subseteq \{A\}_K$, so that $\tau_{D/\D}(g)_{\a}=0$ for all $g\in \I(\a/R)$. Now apply Proposition ~\ref{exten5}.  
\end{proof}

\

\section{Relative prolongations and a characterization of $DCF_{0,m+1}$}

From now on we use freely the basic notions and terminology of model theory; we suggest \cite{Mar0} as a general reference. We work in the language of differential rings $\LL_m=\{0,1,+,-,\times,\d_1,\dots,\d_m\}$, and we let $\LL_0$ be the language of rings. We denote by $DF_{0,m}$ the theory of differential fields of characteristic zero with $m$ commuting derivations, and by $DCF_{0,m}$ its model-completion, the theory of differentially closed fields. In the ordinary case, when $m=1$, we write $DF_0$ and $DCF_0$ in place of $DF_{0,1}$ and $DCF_{0,1}$. The theory of algebraically closed fields of characteristic zero is denoted by $ACF_0$.

Let us recall the geometric axioms of $DCF_0$ given by Pierce and Pillay in \cite{PiPi}. Given a $\d$-field $K$ and $V$ a Zariski-closed set of $K^n$, the prolongation of $V$, $\tau V$, is the Zariski-closed subset of $K^{2n}$ defined by the equations $f=0$ and $\sum_{i=1}^n\frac{\dd f}{\dd x_i}(\x)y_i+f^{\d}(\x)=0$ for each polynomial $f\in K[\x]$ vanishing on $V$. Note that, in terms of our notation from Definition \ref{tauf}, the last equation is just $\tau_{\d/\emptyset}f(\x,\y)=0$.

\begin{fact}[Pierce-Pillay Axioms]\label{PPax}
Let $(K,\d)\models DF_0$. Then $(K,\d)\models DCF_0$ if and only if $K\models ACF_0$ and for each pair of irreducible Zariski-closed sets $V\subseteq K^n$ and $W\subseteq \tau V$ such that $W$ projects dominantly onto $V$, there is $\a\in V$ such that $(\a,\d\a)\in W$.
\end{fact}

The Pierce-Pillay characterization of $DCF_0$ is indeed first-order. Expressing irreducibility of a Zariski-closed set as a definable condition on the parameters uses the existence of bounds to check primality of ideals in polynomial rings in finitely many variables \cite{Van}. Also, if the field is algebraically closed, one can find a first-order formula, in the language of rings, describing for which parameters a Zariski-closed set projects dominantly onto some fixed irreducible Zariski-closed set. This uses the fact that in a model of $ACF_0$ the Morley rank of a definable set agrees with the dimension of its Zariski-closure. 

The goal of this section is to extend the Pierce-Pillay axioms, in an appropriate sense, to the context of several commuting derivations. Our approach is to accomplish this by characterizing $DCF_{0,m+1}$ in terms of the geometry of $DCF_{0,m}$. The Pierce-Pillay axioms are then the $m=0$ case (under the convention $DCF_{0,0}=ACF_0$).

For the rest of this section we fix a differential field $(K,\D\cup\{D\})$ with $\D=\{\d_1,\dots,\d_m\}$, and $V\subseteq K^n$ a $\D$-closed set.

\begin{definition}\label{prolong}
We define the $D/\D$-prolongation of $V$, $\tau_{D/\D}V\subseteq K^{2n}$, to be the $\D$-closed set defined by 
\begin{displaymath}
f=0 \textrm{ and } \tau_{D/\D}f=0, \textrm{ for all } f\in \I(V/K).
\end{displaymath}
Here $\I(V/K)=\{f\in K\{\x\} : f$ vanishes on $V\}$. When $\D$ and $D$ are understood, we just write $\tau f$ and $\tau V$. For $\a\in V$, $\tau (V)_{\a}$ denotes the fibre of $\tau V$ at $\a$. Note that when $m=0$ this is just the usual prolongation.
\end{definition}
By the first equality of Lemma \ref{exten1}, if $\a$ is in $V$ then $(\a,D\a)\in \tau V$. In particular the projection $\pi : \tau V\to V$ given by $\pi(\a,\b)=\a$ is surjective. 

Suppose $A\subseteq K\{\x\}$ is such that $[A]=\I(V/K)$. Then $\tau V$ is defined by $f=0$ and $\tau f=0$ for all $f\in A$. Indeed, this follows from the first argument in the proof of Proposition \ref{exten5}. Moreover, the following consequence of Proposition \ref{better} implies that the $D/\D$-prolongation varies uniformly with $V$.

\begin{proposition}\label{prolim}
Suppose $(K,\D)\models DCF_{0,m}$. If $V=\Z(f_1,\dots,f_s):=\{\a\in K^n : f_i(\a)=0, \, i=1,\dots,s\}$, then $\tau V=\Z(f_i,\tau f_i: i=1,\dots,s)$.
\end{proposition}
\begin{proof}
Clearly $\tau V\subseteq \Z(f_i,\tau f_i: i=1,\dots,s)$. Let $(\a,\b)\in \Z(f_i,\tau f_i: i=1,\dots,s)$. By Proposition \ref{better}, $\tau f(\a,\b)=0$ for all $f\in \{f_1,\dots, f_s\}$. Since $(K,\D)\models DCF_{0,m}$, we have $\{f_1,\dots,f_s\}=\I(\Z(f_1,\dots,f_s))=\I(V)$. Hence, $(\a,\b)\in \tau V$.
\end{proof}

\begin{remark} \label{rem} \
\begin{enumerate}
\item Suppose $(K,\D)\models DCF_{0,m}$. If $V$ is defined over the $D$-constants, that is, $V=\Z(f_1,\dots,f_s)$ where $f_i\in \C_D\{\x\}$, then $\tau V$ is just Kolchin's $\D$-tangent bundle of $V$. Indeed, by Proposition \ref{prolim}, the equations defining $\tau V$ become $f_i(\x)=0$ and $\tau f_i(\x,\y)=df_i(\x)\cdot \t\y=0$, $i=1,\dots,s$. These are exactly the equations for Kolchin's $\D$-tangent bundle $T_\D V$ (\cite{Kol}, Chap.VIII, \S 2). 
\item In general, $\tau V$ is a torsor under $T_\D V$. Indeed, from the equations one sees that $\tau(V)_{\a}$ is a translate of $T_\D(V)_{\a}$, and so the map $T_\D V\times_V \tau V\to \tau V$ given by $((\a,\b),(\a,\c))\mapsto (\a,\b+\c)$ is a regular action of $T_\D V $ on $\tau V$ over ~$V$.
\end{enumerate}
\end{remark}

Note that in case $\D=\emptyset$, part $(2)$ of Remark \ref{rem} reduces to the fact that the prolongation of a Zariski-closed set is a torsor under its tangent bundle.

Here is our extension of the Pierce-Pillay characterization to several commuting derivations.

\begin{theorem}\label{maintheo}
Suppose $(K,\D\cup\{D\})\models DF_{0,m+1}$. Then $(K,\D\cup\{D\})\models DCF_{0,m+1}$ if and only if
\begin{enumerate}
\item $(K,\D)\models DCF_{0,m}$
\item For each pair of irreducible $\D$-closed sets $V\subseteq K^n$, $W\subseteq \tau V$ such that $W$ projects $\D$-dominantly onto $V$. If $O_V$ and $O_W$ are nonempty $\D$-open subsets of $V$ and $W$ respectively, then there exists $\a\in O_V$ such that $(\a,D\a)\in O_W$. 
\end{enumerate}
\end{theorem}

As we will see in the proof, it would have been equivalent in condition (2) to take $O_V=V$ and $O_W=W$. Also note that when $m=0$, the theorem is exactly Fact \ref{PPax}.

\begin{proof}
Suppose $(K,\D\cup\{D\})\models DCF_{0,m+1}$, and $V$, $W$, $O_V$ and $O_W$ are as in condition (2). Let $(\UU,\D)$ be a $|K|^+$-saturated elementary extension of $(K,\D)$. If $X$ is an ($\LL_{m}$-)definable subset of $K^n$, by $X(\UU)$ we mean the interpretation of $X$ in $\UU^n$. Let $(\a,\b)\in \UU^{2n}$ be a $\D$-generic point of $W$ over $K$; that is, $\I(\a,\b/K)=\I(W(\UU)/K)$. Then $(\a,\b)\in O_W(\UU)$. Since $(\a,\b)\in \tau V(\UU)$ we have that $\tau (f)_{\a}(\b)=0$ for all $f\in \I(V(\UU)/K)$. The fact that $W$ projects $\D$-dominantly onto $V$ implies that $\a$ is a $\D$-generic point of $V$ over $K$, so $\a\in O_V(\UU)$ and $\I(\a/K)=\I(V(\UU)/K)$. Hence, $\tau (f)_{\a}(\b)=0$ for all $f\in \I(\a/K)$. By Proposition \ref{exten5}, there is a unique $\D$-derivation $D': K\{\a\}\to \UU$ extending $D$ such that $D' \a=\b$. By Corollary \ref{Null}, we can extend $D'$ to all of $\UU$, call it $D''$. Hence, $\UU$ becomes a $\D \cup \{D''\}$-field extending the $\D\cup\{D\}$-closed field $K$. Since $\a\in O_V(\UU)$, $(\a,\b)\in  O_W(\UU)$ and $D'' \a=\b$, we get a point $(\a',\b')$ in $K$ such that $\a'\in O_V$, $(\a',\b')\in O_W$ and $D\a'=\b'$.

The converse is essentially as in \cite{PiPi}. Let $\phi(\x)$ be a conjunction of atomic $\LL_{m+1}$-formulas over $K$. Suppose $\phi$ has a realisation $\a$ in some $(F,\D\cup\{D\})\models DF_{0,m+1}$ extending of $(K,\D\cup\{D\})$. Let
\begin{displaymath}
\phi(\x)=\psi(\x,\d_{m+1}\x,\dots,\d_{m+1}^r \x)
\end{displaymath}
where $\psi$ is a conjunction of atomic $\LL_m$-formulas over $K$ and $r>0$. Let $\c=(\a,D \a,\dots,D^{r-1}\a)$ and $X\subseteq F^{nr}$ be the $\D$-locus of $\c$ over $K$. Let $Y\subseteq F^{2nr}$ be the $\D$-locus of $(\c,D \c)$ over $K$. Let
\begin{displaymath}
\chi(\x_0,\dots,\x_{r-1},\y_0,\dots,\y_{r-1}) :=\psi(\x_0,\dots,\x_{r-1},\y_{r-1}) \land \left(\land_{i=1}^{r-1}\x_i = \y_{i-1}\right)
\end{displaymath}
then $\chi$ is realised by $(\c,D\c)$. Since $(\c,D\c)$ is a $\D$-generic point of $Y$ over $K$ and its projection $\c$ is a $\D$-generic point of $X$ over $K$, we have that $Y$ projects $\D$-dominantly onto $X$ over $K$. Thus, since $(K,\D)\models DCF_{0,m}$, $Y(K)$ projects $\D$-dominantly onto $X(K)$. Also, since $(\c,D \c)\in\tau X$, we have $Y(K)\subseteq\tau (X(K))$. Applying (2) with $V=O_V=X(K)$ and $W=O_W=Y(K)$, there is $\bar d$ in $V$ such that $(\bar d,D\bar d)\in W$. Let $\bar d=(\bar d_0,\dots,\bar d_{r-1})$ then $(\bar d_0,\dots,\bar d_{r-1},D\bar d_0,\dots,D \bar d_{r-1})$ realises $\chi$. Thus, $(\bar d_0, D \bar d_0,\dots,D^r \bar d_0)$ realises $\psi$. Hence, $\bar d_0$ is a tuple of $K$ realising $\phi$. This proves that $(K,\D\cup\{D\})\models ~DCF_{0,m+1}$.
\end{proof}

\

\section{Making the geometric characterization first-order}

It is not known to the author if condition (2) of Theorem \ref{maintheo} can be expressed in a first-order way. One issue is to express irreducibility of $\D$-closed sets as a definable condition. This seems to be an open problem related to the generalized Ritt problem \cite{Ov}. The other issue is how to express when a $\D$-closed set projects $\D$-dominantly onto another $\D$-closed set as a definable condition. Unlike the algebraic case, in differentially closed fields, Morley rank does not concide with the Krull-Noetherian geometric dimension (see the example given in \S 2.5 of \cite{Mar2}). 

We resolve the problem in this section by modifying the characterization of Theorem \ref{maintheo} so that it no longer mentions irreducibility or dominance. The first of these can be handled rather easily by the following lemma. 

\begin{lemma}\label{star}
Let $K$ be a $\D\cup\{D\}$-field. Let $V\subseteq K^n$ be a $\D$-closed set with $K$-irreducible components $\{V_1,\dots,V_s\}$. If $\a \in V_i\backslash\bigcup_{j\neq i}V_j$, then $\tau(V)_{\a}=\tau(V_i)_{\a}$.
\end{lemma}
\begin{proof}
Clearly $\tau(V_i)_{\a}\subseteq\tau(V)_{\a}$. Let $\b\in\tau(V)_{\a}$ and $f\in \I(V_i/K)$. Since $\a$ is not in $V_j$, for $j\neq i$, we can pick a $g_j\in \I(V_j/K)$ such that $g_j(\a)\neq 0$. Then, if $g=\prod_j g_j$, we get $fg\in \I(V/K)$ and so
\begin{displaymath}
0= \tau(fg)_{\a}(\b) = \tau(f)_{\a}(\b)g(\a)+f(\a)\tau(g)_{\a}(\b) = \tau(f)_{\a}(\b)g(\a)
\end{displaymath}
where the third equality holds because $\a\in V_i$. Since $g(\a)\neq 0$, we have $\tau(f)_{\a}(\b)=0$, and so $\b\in\tau(V_i)_{\a}$.
\end{proof}

It follows that if $W\subseteq \tau V$ projects $\D$-dominantly onto $V$ and $V_i$ is a $K$-irreducible component of $V$, then a $K$-irreducible component of $W\cap \tau V_i$ projects $\D$-dominantly onto $V_i$.

The second issue, that of $\D$-dominant projections, is more difficult to deal with. Let us note here that when $\D=\emptyset$, that is, in the case of $DCF_0$, one can just replace dominant projections by surjective projections in the Pierce-Pillay axiomatization. Indeed this reformulation is stated in \cite{Pi}. We will not give a proof here as it will follow from Theorem \ref{main2} below. However, what makes this work is the fact that if $a$ is $D$-algebraic over $K$, then $D^{k+1} a \in K(a,D a, \dots, D^k a)$ for some $k$. In several derivations it is not necessarily the case that if $a$ is $\D\cup\{D\}$-algebraic over $K$, then $D^{k+1} a$ is in the $\D$-field generated by $a,Da,\dots,D^k a$ over $K$, for some $k$. However, by a theorem of Kolchin (Fact \ref{good1} below), this can always be achieved if we allow $\QQ$-linear transformations of the derivations. Our modification of Theorem \ref{maintheo} will therefore need to refer to such transformations.

For every $M=(c_{i,j})\in GL_{m+1}(\QQ)$, let $\D'=\{\d'_1,\dots,\d'_m\}$ and $D'$ be the derivations on $K$ defined by $\d'_i= c_{i,1}\d_1+\dots+c_{i,m}\d_m+c_{i,m+1}D$ and $D'=c_{m+1,1}\d_1+\dots+c_{m+1,m}\d_m+c_{m+1,m+1}D$. In this case we write $(\D',D')=M(\D,D)$. Clearly, the elements of $\D'\cup\{D'\}$ are also commuting derivations on $K$.

\begin{fact}[\cite{Ko}, Chap. II, \S 11]\label{good1}
Let $(K,\D\cup\{D\})\models DF_{0,m+1}$. Let $\a=(a_1,\dots,a_n)$ be a tuple of a $\D\cup\{D\}$-field extension of $K$. Suppose all the $a_i$'s are $\D\cup\{D\}$-algebraic over $K$, then there exists $k>0$ and a matrix $M\in GL_{m+1}(\QQ)$ such that, writing $(\D',D')=M(\D,D)$, we have that $D^{\ell}\a$ is in the $\D'$-field generated by $\a,D'\a\dots,D'^k\a$ over $K$, for all $\ell>k$.
\end{fact}

Theorem \ref{maintheo} characterizes $DCF_{0,m+1}$ in terms of the geometry of $DCF_{0,m}$. The idea, of course, was that $DCF_{0,m}$ has a similar characterization relative to $DCF_{0,m-1}$, and so on. In our final formulation (Theorem \ref{main2} below) we will implement this recursion and give, once and for all, a geometric first-order axiomatization of $DCF_{0,m+1}$ for all $m\geq 0$, that refers only to the base theory $ACF_0$.

\begin{theorem}\label{main2}
Suppose $(K,\D\cup\{D\})\models DF_{0,m+1}$. Then $(K,\D\cup\{D\})\models DCF_{0,m+1}$ if and only if 
\begin{enumerate}
\item $K\models ACF_0$
\item Suppose $M\in GL_{m+1}(\QQ)$, $(\D',D'):=M(\D,D)$, $V=\Z(f_1,\dots,f_s)\subseteq K^n$ is a nonempty $\D'$-closed set,
and 
\begin{displaymath}
W\subseteq \Z(f_1,\dots,f_s,\tau_{D'/\D'}f_1,\dots,\tau_{D'/\D'}f_s)\subseteq K^{2n}
\end{displaymath}
is a $\D'$-closed set that projects onto $V$. Then there is $\a\in V$ such that $(\a,D'\a)\in W$.
\end{enumerate}
\end{theorem}
\begin{proof}
Suppose $(K,\D\cup\{D\})\models DCF_{0,m+1}$. Clearly $K\models ACF_0$. Suppose $M$, $\D'$, $V$ and $W$ are as in condition (2). Clearly $(K,\D'\cup\{D'\})\models DCF_{0,m+1}$, so by Proposition \ref{prolim} we have that $\Z(f_i,\tau_{D'/\D'}f_i:i=1,\dots,s)=\tau_{D'/\D'}V$. Let $V_i$ be an irreducible component of $V$ and $W'=W\cap \tau_{D'/\D'}V_i$. By Lemma \ref{star}, we can find an irreducible component of $W'$ projecting $\D'$-dominantly onto $V_i$. Now just apply Theorem \ref{maintheo} (with $\D'\cup\{D'\}$ rather than $\D\cup\{D\}$) to get the desired point.

For the converse, we assume conditions (1) and (2) and prove that $(K,\D\cup\{D\})\models DCF_{0,m+1}$. Given $r=1,\dots,m+1$ and $N\in GL_{m+1}(\QQ)$, let $\KK_{r,N}=(K,\bar \D_{r-1}\cup\{\bar D\})$ where $(\bar\D,\bar D)=N(\D,D)$ and $\bar\D_{r-1}=\{\bar \d_1,\dots,\bar \d_{r-1}\}$. Set $\KK_{0,N}$ to be the pure algebraic field $K$. We show by induction that for each $r=0,\dots,m+1$, $\KK_{r,N}\models DCF_{0,r}$ for all $N\in GL_{m+1}(\QQ)$. The result will then follow by setting $r=m+1$ and $N=\operatorname{id}$. The case of $r=0$ is just assumption (1). We assume $0\leq r\leq m$, $N\in GL_{m+1}(\QQ)$, and we show that $\KK_{r+1,N}=(K,\bar \D_r\cup\{\bar D\})\models DCF_{0,r+1}$.

Suppose $\phi(\x)$ is a conjunction of atomic $\LL_{r+1}$-formulas over $K$, with a realisation $\a=(a_1,\dots,a_n)$ in some $\bar\D_{r}\cup\{\bar D\}$-field $F$ extending $\KK_{r+1,N}$. We need to find a realisation of $\phi$ in $\KK_{r+1,N}$. We may assume that each $a_i$ is $\bar\D_{r}\cup\{\bar D\}$-algebraic over $K$ (this can be seen algebraically or one can use the existence of prime models of $DCF_{0,r+1}$ over $K$, see \S 3.2 of \cite{Mc}). 

Let $M'\in GL_{r+1}(\QQ)$ and $k>0$ be the matrix and natural number given by Fact \ref{good1}. Let $M\in GL_{m+1}(\QQ)$ be
\begin{displaymath}
M=E\left(
\begin{array}{c c}
M' & 0 \\
0 & I
\end{array}
\right) EN
\end{displaymath}
where $E$ is the elementary matrix of size $(m+1)$ that interchanges row $(r+1)$ with row $(m+1)$ and $I$ is the identity matrix of size $(m-r)$. Then, setting $(\D',D')=M(\D,D)$, we get
\begin{equation}\label{large}
D'^{k+1}\a = \frac{f(\a, D'\a,\dots, D'^k\a)}{g(\a,D'\a\dots,D'^{k}\a)}
\end{equation}
for some $f, g\in (K\{\x_0,\dots,\x_{k}\}_{\D_r'})^n$. Here $\D'_r=\{\d'_1,\dots,\d'_r\}$ and $K\{\x\}_{\D'_r}$ denotes the $\D'_r$-ring of $\D'_r$-polynomials over $K$. Let
\begin{displaymath}
\c=\left(\a,D' \a,\dots,D'^k\a,\frac{1}{g(\a,D'\a\dots,D'^{k}\a)}\right). 
\end{displaymath}
Let $X\subseteq F^{n(k+2)}$ be the $\D'_r$-locus of $\c$ over $K$ and $Y\subseteq F^{2n(k+2)}$ the $\D_r'$-locus of $(\c,D'\c)$ over $K$.

\noindent {\bf Claim.} $Y$ projects onto $X$.

\noindent Consider the $\D_r'$-polynomial map $s(\x_0,\dots,\x_{k+1}):X\to F^{n(k+2)}$ given by
\begin{displaymath}
s=(\x_1,\dots,\x_k,f \, \x_{k+1},-\x_{k+1}^2\tau_{D'/\D_r'}g(\x_0,\dots,\x_k,\x_1,\dots,\x_k,f\,\x_{k+1}))
\end{displaymath}
where any product between tuples is computed coordinatewise. Using (\ref{large}), an easy computation shows $s(\c)=D'\c$. Given $\b\in X$, we note that $(\b,s(\b))\in Y$. Indeed, if $h$ is a $\D_r'$-polynomial over $K$ vanishing at $(\c,D'\c)$, then $h(\cdot,s(\cdot))$ vanishes at $\c$ and hence on all of $X$. So $(\b,s(\b))$ is in the $\D_r'$-locus of $(\c,D'\c)$ over $K$. That is, $(\b,s(\b))\in Y$. As this point projects onto $\b$ we have proven the claim.

Now, by induction, $(K,\D'_r)\models DCF_{0,r}$. Indeed, $(K,\D'_r)=\KK_{r,N'}$ where $N'$ is obtained from $M$ by interchanging rows $r$ and $(m+1)$. Hence, the claim implies that $Y(K)$ projects onto $X(K)$. Also, if $X(K)=\Z(f_1,\dots,f_s)$ where each $f_i$ is a $\D'_r$-polynomial, then clearly $Y(K)\subseteq \Z(f_i,\tau_{D'/\D'_r}f_i:i=1,\dots,s)$. Hence, by condition (2), there is $\bar d\in X(K)$ such that $(\bar d,D'\bar d)\in Y(K)$.

Now, let $\rho(\x)$ be the $\LL_{r+1}$-formula over $K$ obtained from $\phi$ by replacing each $\d_1,\dots, \d_{r+1}$ for $d_{i,1}\d_1+\cdots+d_{i,r+1}\d_{r+1}$, where $(d_{i,j})\in GL_{r+1}(\QQ)$ is the inverse matrix of $M'$. By construction, $\phi^{(K,\bar\D_{r}\cup\{\bar D\})}=\rho^{(K,\D'_{r}\cup\{D'\})}$. Thus it suffices to find a realisation of $\rho$ in $(K,\D'_{r}\cup\{D'\})$. We may assume that the $k$ of (\ref{large}) is large enough so that we can write
\begin{displaymath}
\rho(\x)=\psi(\x,\d_{r+1}\x,\dots,\d_{r+1}^k\x)
\end{displaymath}
where $\psi$ is a conjunction of atomic $\LL_r$-formulas over $K$. Let
\begin{displaymath}
\chi(\x_0,\dots,\x_{k+1},\y_0,\dots,\y_{k+1}):=\psi(\x_0,\dots,\x_k)\land\left(\land_{i=1}^k x_i=y_{i-1} \right).
\end{displaymath}
Then $(F,\D'_r)\models \chi(\c,D'\c)$, and so, as $(\bar d,D'\bar d)$ is in the $\D_r'$-locus of $(\c,D'\c)$ over $K$, we have that $(F,\D'_r)\models \chi(\bar d,D'\bar d)$. But since $\bar d$ is a $K$-point, we get $(K,\D'_r)\models \chi(\bar d,D'\bar d)$. Writing the tuple $\bar d$ as $(\bar d_0,\dots,\bar d_{r+1})$, we see that $\bar d_0$ is a realisation of $\rho$ in $(K,\D'_{r}\cup\{D'\})$. This completes the proof.
\end{proof}

\begin{remark}\
\begin{enumerate}
\item Condition (2) of Theorem \ref{main2} is indeed first-order; expressible by an infinite collection of $\LL_{m+1}$-sentences, one for each fixed choice of $M$, $f_1,\dots,f_s$ and ``shape'' of $W$.
\item In condition (2) we can strengthen the conclusion to ask for $\{\a\in V: (\a,D'\a)\in W\}$ to be $\D'$-dense in $V$.
\end{enumerate}
\end{remark}

\

%\bibliographystyle{plain}
%\bibliography{ccs}

\begin{thebibliography}{10}

\bibitem{Bl}
L. Blum.
\newblock Differentially Closed Fields: A Model Theoretic Tour.
\newblock Contributions to Algebra, Academic Press Inc (1977).

\bibitem{Ko}
E. Kolchin.
\newblock Differential Algebra and Algebraic Groups.
\newblock Academic Press. New York, New York (1973).

\bibitem{Kol}
E. Kolchin.
\newblock Differential Algebraic Groups.
\newblock Academic Press (1985).

\bibitem{La}
S. Lang.
\newblock Algebra.
\newblock Springer-Verlag. Third Edition (2002).

\bibitem{Mar0}
D. Marker.
\newblock Model Theory: an Introduction.
\newblock Springer-Verlag. Third Edition (2002).

\bibitem{Mar2}
D. Marker, M. Messmer and A. Pillay.
\newblock Model Theory of Fields.
\newblock Association for Symbolic Logic. LNL 5. Second Edition (2006).

\bibitem{Mc}
T. McGrail.
\newblock The Model Theory of Differential Fields with Finitely Many Commuting Derivations.
\newblock The Journal of Symbolic Logic. Vol. 65, No. 2, pp. 885-913 (June 2000).

\bibitem{Ov}
O. Golubitsky, M. Kondratieva and A. Ovchinnikov.
\newblock On the Generalised Ritt Problem as a Computational Problem.
\newblock Preprint (2010).

\bibitem{Pie}
D. Pierce.
\newblock Fields with Several Commuting Derivations.
\newblock Preprint (2010).

\bibitem{PiPi}
D. Pierce and A. Pillay.
\newblock A Note on the Axioms for Differentially Closed Fields of Characteristic Zero.
\newblock Journal of Algebra 204, 108-115 (1998).

\bibitem{Pi}
A. Pillay.
\newblock Model Theory and Stability Theory, with Applications in Differential Algebra and Algebraic Geometry.
\newblock . Model Theory with Applications to Algebra and Analysis. Vol. 1. London Mathematical Society, Lecture Note Series 349, pp. 1-24 (2008).

\bibitem{Tr}
M. Tressl.
\newblock The Uniform Companion for Large Differential Fields of Characteristic 0.
\newblock Transactions of the American Mathematical Society. Vol. 357, No. 10, pp. 3933-3951 (2005).

\bibitem{Van}
L. van den Dries and K. Schmidt.
\newblock Bounds in the Theory of Polynomial Rings over Fields. A Nonstandard Approach.
\newblock Inventiones Mathematicae 76, pp. 77-91 (1984).

\end{thebibliography}

\end{document}